\documentclass{psp}

\let\realverbatim=\verbatim
\let\realendverbatim=\endverbatim
\renewcommand\verbatim{\par\addvspace{6pt plus 2pt minus 1pt}\realverbatim}
\renewcommand\endverbatim{\realendverbatim\addvspace{6pt plus 2pt minus 1pt}}
\makeatletter
\newcommand\verbsize{\@setfontsize\verbsize{10}\@xiipt}
\renewcommand\verbatim@font{\verbsize\normalfont\ttfamily}
\makeatother

\usepackage{amscd}
\usepackage{amsmath}
\usepackage{amsfonts}
\usepackage{amssymb}
\usepackage[all]{xy}
\usepackage{graphicx}
\usepackage{verbatim}
\usepackage{color}
\usepackage{enumerate}
\usepackage{MnSymbol}
\usepackage{calrsfs}

\usepackage{diagrams}

\def\c{{\bf c}}
\def\D{{\bf D}}
\def\K{{\bf K}}
\def\T{{\bf T}}
\def\p{{p}}

\def\QQ{{Q}}

\DeclareMathAlphabet{\pazocal}{OMS}{zplm}{m}{n}

\newtheorem{theorem}{Theorem}[section]
\newtheorem{lemma}[theorem]{Lemma}
\newtheorem{proposition}[theorem]{Proposition}
\newtheorem{corollary}[theorem]{Corollary}

\newtheorem{definition}[theorem]{Definition}
\newtheorem{example}[theorem]{Example}

\newtheorem{remark}[theorem]{Remark}

\usepackage{amsmath}
\usepackage{url}

\def\Jac{\mbox{\rm Jac$\:$}}

\def\Q{{Q}}

\def\ff{\frak}
\def\Spec{\mbox{\rm Spec}}

\def\Hom{\mbox{ Hom}}

\def\Ker{\mbox{ Ker }}

\def\End{\mbox{\rm End}}

\def\p{{\rm{p}}}

\begin{document}

\title[Krull's Principal Ideal Theorem]
  {Krull's Principal Ideal Theorem in  \\ non-Noetherian settings}
\author[Bruce Olberding]{Bruce Olberding\\
Department of Mathematical Sciences, New Mexico State University, \\
Las Cruces, NM 88003-8001
\addressbreak
  e-mail\textup{: \texttt{olberdin@nmsu.edu}}}

\receivedline{Received \textup{15} December \textup{2016}}

\maketitle

\begin{abstract}
Let $P$ be a finitely generated ideal of a commutative ring $R$. Krull's Principal Ideal Theorem states that if $R$ is Noetherian and  $P$ is minimal over a principal ideal of $R$, then $P$ has height at most one.  Straightforward examples show that this assertion fails if $R$ is not Noetherian. We consider what can be asserted in the non-Noetherian case in place of Krull's theorem. MSC 2010: 13C15, 13A15

\end{abstract}



\section{Introduction}

Let $R$ be a Noetherian commutative ring. Krull's Principal Ideal Theorem (PIT)   states that  a  prime ideal minimal over a principal ideal of $R$ has height at most one.   It is easy to find examples of non-Noetherian local rings  having a maximal ideal  of height more than one and  minimal over a principal ideal.  
 A simple example is given by  any valuation domain of Krull dimension at least two whose maximal ideal is not the union of the prime ideals properly contained in it.

Moving beyond Noetherian rings,  Krull's proof of the PIT can be adapted to 
assert a  height criterion for rings in the case in which $R$ is not necessarily Noetherian.   For ideals $I$ and $Q$ of the ring $R$ with $Q$  prime, we let $I_{(Q)} = \{r \in R:br \in I$ for some $b \in R \setminus Q\}$.

\medskip

\noindent
{\bf Height Zero Criterion.}
  {\it Let $R$ be a ring, and let $P$ be a prime ideal of $R$ such that $PR_P$ is a finitely generated ideal and $P$ is minimal over a principal ideal of $R$. Let  $Q$ be a prime ideal of $R$ with  $Q \subsetneq P$. Then {\rm ht}$(Q) =0$ if and only if 
$Q$ is minimal over a finitely generated ideal $I$ such that 
 $(I^n)_{(Q)}/(I^{n+1})_{(Q)}$ is a finitely generated $R$-module for each $n>0$.}

\medskip

The PIT for Noetherian rings follows easily from this criterion. 
The condition that $(I^n)_{(Q)}/(I^{n+1})_{(Q)}$ is a finitely generated $R$-module, which obviously holds in the Noetherian case (take $I = Q$), is what proves difficult to satisfy in  non-Noetherian rings.

Although height criteria will not be the focus of this article, we mention this adaptation in order to
help pinpoint technical aspects of why the PIT breaks down for non-Noetherian rings. We also do so to
 contrast the elementary proof of the PIT with the arguments in this article, which are more involved because they rely on tools from local algebra and multiplicative ideal theory that have been developed since Krull's time. In this sense, 
the present article can be viewed as a revisiting of Krull's PIT from a contemporary point of view with an effort to find affirmative assertions  for finitely generated prime ideals minimal over a principal ideal, regardless of whether the ring is Noetherian.

In order to highlight how Krull's PIT follows from first principles, we give a quick proof of the  Height Zero Criterion  that is based on Krull's original argument for the Noetherian case and which can be found in Krull  \cite[\S 3]{Krull} or Northcott \cite[Theorem 6, p.~58]{Northcott}; see also  Nagata \cite[(9.2), p.~26]{Na}:
 If ht$(Q) = 0$, then  $I = (0)$ satisfies the conclusion of the criterion. 
Conversely, suppose that $(I^n)_{(Q)}/(I^{n+1})_{(Q)}$ is a finitely generated $R$-module for each $n>0$. Without loss of generality $R$ is local with maximal ideal $P$. By assumption, there is $x \in P$ such that $P = \sqrt{xR}$. For each $n >0$, let $I_n = (I^n)_{(Q)}$.    
Since $R/xR$ is an Artinian ring, there is $n>0$ such that $I_n +xR = I_{n+1} + xR$. Therefore, $I_n \subseteq I_{n+1} + xR$.  Let $a \in I_n$. Then there is $b \in I_{n+1}$ and $r \in R$ such that $a = b + xr$.  Thus, $xr = a-b \in I_n$. Since $x \not \in Q$, it follows that $r \in I_n$.  This shows that $ I_n \subseteq I_{n+1} + xI_n$, and hence $I_n  = I_{n+1}+PI_n$. Thus $P(I_n/I_{n+1}) = I_n/I_{n+1}$, and since $I_n/I_{n+1}$ is a finitely generated $R$-module, 
 Nakayama's Lemma implies $I_n = I_{n+1}$. 
  Thus $I^nR_Q= I^{n+1}R_Q$. Since $I^nR_Q$ is a finitely generated proper ideal of $R_Q$, another application of Nakayama's Lemma shows that 
$I^nR_Q = 0$.  Since $Q$ is minimal over $I$, this implies ht$(Q) = 0$.

Our focus in the present article is not on a further analysis of  the 
height zero ideals contained in a prime ideal minimal over a principal ideal, nor do we  consider height theorems for specific classes of rings (but see  \cite{ADEH, BAD, KM} for work in such a direction).   
 Instead, we consider the  general question of what can be said affirmatively about a finitely generated prime ideal that is minimal over a principal ideal when no additional restrictions are placed on the ring.  
Some motivating examples for how such a situation arises include local rings with finitely generated maximal ideal and   finite prime spectrum,  stable rings \cite{OFS},  ultrapowers of one-dimensional local Noetherian rings \cite{OS} and  the ring of integer-valued polynomials on a subset of the ring of $p$-adic integers \cite{CHL} (see Example~\ref{new examples}).

  A few reductions are helpful.
 
 \smallskip
 
 {\it First reduction.} Using standard localization arguments, we can assume that the ring $R$ is a local ring and $P$ is the maximal ideal of $R$. To emphasize this, we use ``$M$'' instead of ``$P$'' for this maximal ideal.  We are then interested in the local rings for which there is a principal ideal that is primary for the finitely generated maximal ideal $M$.  For this reason, all our results concern {\it local rings with finitely generated maximal ideal. }   

 \smallskip

Using Sally's work on the Hilbert functions of low-dimensional local Noetherian rings, as well as an approach to reductions of ideals due to Eakin-Sathaye \cite{ES}, we first prove in Theorem~\ref{pre pre stable} 
 that  a finitely generated maximal ideal $M$ of a local ring $R$  
 is minimal over   a principal ideal of $R$ if and only if there is $n>0$ such that every $M$-primary ideal of $R$ can be generated by $n$ elements.


\smallskip
 
 {\it Second reduction.} Let $R_{red}$ be the reduced quotient of the local ring $R$, i.e., $R_{red} = R/N(R)$, where $N(R)$ is the nilradical of $R$.  If the Krull dimension of $R$ is at least one, then $M$ is the radical of a principal ideal if and only if there is a nonzerodivisor $x$ in $R_{red}$ such that $MR_{red} = \sqrt{xR_{red}}$.  This motivates our primary focus in Section 3: What can be asserted about  a local ring $R$ with finitely generated maximal ideal that is the radical of a principal regular\footnote{An ideal of a commutative ring is {\it regular} if it contains a nonzerodivisor.} ${}$ ideal? Equivalently, we are interested in what can be asserted about {\it a local ring having a finitely generated regular maximal ideal that is the radical of a principal ideal.}

 \smallskip

In Theorem~\ref{pre stable} we characterize in a number of ways the situation in which a finitely generated regular  maximal ideal of a local ring is the radical of a principal ideal.  For example, we show that this occurs if and only if there is an integral overring of $R$ in which every maximal ideal is a principal ideal.


\smallskip
 
 {\it Third reduction.}  Following Cohen \cite{Cohen}, a local $R$ with maximal ideal $M$ is a {\it generalized local ring}\footnote{For Cohen a local ring 
was necessarily Noetherian, and hence in his setting every local ring is a generalized local ring. Although we use his  terminology of ``generalized local ring,''  since we do not assume a local ring is Noetherian, it follows that a local ring for us need not be a generalized local ring. }
$\:$ if 
$M$ is finitely generated and $\bigcap_{i>0}M^i = 0$.  If $R$ is a local ring with finitely generated maximal ideal $M$, then $R/(\bigcap_{i>0}M^i)$ is a generalized local ring. This motivates the third reduction of our question: What can be asserted about {\it a generalized local ring with regular maximal ideal that is the radical of a principal ideal}?

 \smallskip
 
 To further justify this reduction, we show in Theorem~\ref{glr theorem} that a local ring  $R$ with regular maximal ideal that is  the radical of  a principal ideal is a pullback of a generalized local ring along a factor map of a semilocal ring. This can be viewed 
 as a decomposition result for $R$, one which makes precise the reduction to a generalized local ring. Examples due to Cahen, Houston and Lucas \cite{CHL}, Gabelli and Roitman \cite{GR} and Heinzer, Rotthaus and Wiegand \cite{HRW} show that a generalized local domain whose maximal ideal is the radical of a principal ideal need not have Krull dimension one; see Example~\ref{new examples}. 
 
With this  last reduction, 
 we find finally in Theorem~\ref{at least one}    a manifestation of the height one nature of the PIT by showing that for a generalized local ring with regular maximal ideal that is the radical of a principal ideal, there is an integral overring of $R$ that has a height one maximal ideal. 
 In contrast to the Noetherian case, the height one prime ideal is  revealed  in an integral extension rather than in the ring itself. 
The examples given in \cite{CHL}  and  \cite{GR} of generalized local rings $R$  of dimension more than one whose maximal ideal is minimal over a principal regular ideal   were constructed as subrings of intersections of a one-dimensional local ring and a semilocal ring. Theorem~\ref{at least one} shows that all such examples must arise this way.

Although our rings are not necessarily Noetherian, we use throughout the article techniques from local algebra. One way in which this is made possible is through a theorem of Cohen  \cite[Theorem 3]{Cohen}   that implies that the $M$-adic completion of a local ring with finitely generated maximal ideal  $M$ is a Noetherian ring. For a different application to non-Noetherian local rings based on passage to the $M$-adic completion, see Schoutens \cite{Sch}. In Section 2, which deals with some needed preliminaries, we use techniques from multiplicative ideal theory in order to develop some technical properties of integrally closed rings that are needed in the later sections. 

\medskip

{\it Notation and terminology.} The {\it dimension} of an ideal $I$ of the ring $R$ is the dimension of the ring $R/I$.  
We denote by $\Q(R)$ the total quotient ring of $R$, and by $\overline{R}$ the integral closure of $R$ in $\Q(R)$. An {\it overring} of $R$ is a ring between $R$ and $\Q(R)$.  
If $I$ and $J$ are $R$-submodules of $\QQ(R)$, we denote by  $(J:_{\QQ(R)} I)$ the $R$-submodule  $\{q \in \QQ(R):qI \subseteq J\}$ of $\QQ(R)$.    
 The Jacobson radical of $R$ is denoted $\Jac R$.
 A ring $R$ is {\it local} if it has a unique maximal ideal. In particular, we do not assume a local ring is Noetherian.
  If $R$ is local with maximal ideal $M$, then $\widehat{R}$ denotes the $M$-adic completion of $R$.




\section{Rings that are Pr\"ufer  in dimension zero}


In the next section we examine integral extensions of a local ring $R$ with finitely generated regular maximal ideal that is the radical of a principal ideal. We show in Theorem~\ref{pre stable} that by taking the integral closure $R'$ in a particular overring of such a ring $R$, we obtain a ring for which every zero-dimensional ideal is invertible. In this section we develop some of the multiplicative ideal theory needed to prove this result by examining semilocal rings for which every finitely generated zero-dimensional ideal is invertible. 
A {\it Pr\"ufer ring} is a ring for which every finitely generated regular ideal $I$ is invertible; i.e., $II^{-1} = R$, where $I^{-1} = (R:_{Q(R)}I)$.   We say a ring $R$ is {\it Pr\"ufer in dimension zero} if every finitely generated regular zero-dimensional ideal is invertible. 
In this section we consider semilocal rings that are Pr\"ufer in dimension zero. 
These are the sorts of rings that arise as integral closures in the next section.



For an extension
 $R \subseteq T$   of rings,  we
 write $T \subseteq \Q(R)$
 if for each $t \in T$ there exists a nonzerodivisor $b \in R$ such
 that $bt \in R$.  
 {If $C$ is an ideal of a ring $R$, then $x \in R$ is
{\it prime to $C$} if whenever $yx \in C$ for some $y \in R$, then
$y \in C$ (i.e., $x + C$ is a nonzerodivisor of $R/C$).


\begin{lemma} \label{n generator} Let $R$ be a semilocal ring such that $\Jac R = \sqrt{mR}$ for some nonzerodivisor $m \in R$, let $C = \bigcap_{k>0}m^kR$ and let $n >0$. Then a zero-dimensional ideal $I$ of $R$  can be generated by $n$ elements if and only if $I/C$ can be generated by $n$ elements. 
\end{lemma}

\begin{proof}  Let $T = R[1/m]$, and observe that $C$ is an ideal of $T$.   
It is clear that if $I$ can be generated by $n$ elements, then so can $I/C$.  Conversely, suppose that $I/C$ can be generated by $n$ elements, say $I = J + C$, where $J $ is an ideal of $R$ that can be generated by $n$ elements.  Since $I$ is zero-dimensional and $C$ is an ideal of $T$, we have $T = IT = JT + C$. 

We claim that $C \subseteq \Jac T$.   If $C \not \subseteq \Jac T$, then  there exists a maximal ideal $N$ of $T$ such that $C + N = T$, and hence $1 = c + n$ for some $c \in C$ and $n \in N$.  Since $C \subseteq \Jac R$, 
 $n=1-c$ is a unit in $R$, hence a unit in $T$, a contradiction that shows that $C \subseteq \Jac T$.  
Since $C \subseteq \Jac T$ and $T = JT + C$, Nakayama's Lemma implies that
$T = JT$. 

We claim next that 
 $C \subseteq J$.  Since $JT = T$, there exist $j_1,\ldots,j_n \in J$
and $t_1,\ldots,t_n \in T$ such that $1 = j_1t_1 + \cdots + j_nt_n$.
If $c \in C$, then since $CT = C \subseteq R$ we have $$c = j_1(t_1c) + \cdots + j_n(t_nc) \in
 J.$$ Hence $C \subseteq J$. Since $I = J+C$, we conclude that $I = J$,  
 which proves the lemma.
%
%
%
\end{proof} 



 \begin{theorem} \label{new lemma} Let $R$ be a semilocal ring that is Pr\"ufer in dimension zero, and suppose   $\Jac R $ is a regular ideal and  the radical of a finitely generated ideal $A$.   Let  $R^* = R/(\bigcap_{k>0}A^k)$.  Then
  \begin{itemize}
  \item[(1)] Each maximal ideal of $R$ contains a unique largest nonmaximal prime ideal.
    \item[(2)]  $R^*$ is  a reduced ring with finitely many minimal prime ideals.
    \item[(3)] $R^*$ is Pr\"ufer in dimension zero. 
  
  \item[(4)] Each minimal prime ideal of $R^*$  is contained in a height one maximal ideal of $R^*$.  
  
  \item[(5)] Each height one maximal ideal of $R^*$ contains only one nonmaximal prime ideal.

  \end{itemize}
\end{theorem}

%
 %
 
 \begin{proof}  Since
$R$ is Pr\"ufer in dimension $0$ and $A$ is a finitely generated regular zero-dimensional ideal of $R$, $A$ is an invertible ideal of $R$. Since $R$ is semilocal, $A = mR$ for some $m \in A$.   
 
(1)  Let $T = R[1/m]$.
 Since an ideal $I$ of $R$ is zero-dimensional if and only if  $IT = 
 T$, every finitely generated ideal $I$ of $R$ for which $IT= T$ is invertible, so that in the terminology of \cite{KZ}, $R \subseteq T$ is a Pr\"ufer extension. Therefore, by \cite[Theorem 2.11 and Lemma 2.12, pp.~102--104]{KZ}, we have that for each maximal ideal $M$ of $R$ and pair of  finitely generated ideals $I$ and $J$ of $R$ with $J$ zero-dimensional,  either $IR_M \subseteq JR_M$ or $JR_M \subseteq IR_M$.  We use this fact next.

 Let $M$ be a maximal ideal of $R$. 
 We show that there is a principal ideal $J$ of $R$ such that $M = \sqrt{J}$.  
 Let $M_1,\ldots,M_n$ denote the maximal ideals of $R$ and assume that $M = M_1$.     
Choose $y \in M \setminus (\bigcup_{j >1}M_j)$.  Then $\sqrt{A+yR} = M$. 
 Since $M$ contains a nonzerodivisor, so does $J := A+yR$.  
 The fact that  $R$ is Pr\"ufer in dimension zero implies that the finitely generated regular ideal  $J$ is an invertible ideal of $R$. Since $R$ is semilocal and $J$ is invertible, $J$ is principal, which proves the claim.

 We claim that $P:=  \bigcap_{k>0} J^k$ is a prime ideal of $R$. 
 Let $r,s \in R$ such that $rs \in P$. 
Suppose
that $r \not \in P$ and $s \not \in P$.  
 Then there exists $k>0$ such that $r,s \not \in J^k$.  
 Then $r/1 \not \in J^kR_M$ since otherwise the fact that  $J$ is $M$-primary implies that $r \in J^k$, a contradiction.   Similarly, $s/1 \not \in J^kR_M$.  As noted, each finitely generated ideal of $R_M$ is comparable to $J^kR_M$. Therefore, 
 $J^kR_M \subseteq rR_M \cap sR_M$, which implies $J^{2k}R_M \subseteq rsR_M \subseteq PR_M \subseteq J^{2k+1}R_M$. Since $J$ is a principal regular ideal, this implies that $JR_M = R_M$, a contradiction to the fact that $J \subseteq M$. Therefore, $P$ is a prime ideal of $R$.  

Next we claim that every prime ideal of $R$ properly contained in $M$ is contained in $P$. 
Let $Q$ be a prime ideal of $R$   contained in $M$. Each finitely generated ideal contained in $QR_M$, and hence $QR_M$ itself, is comparable to $J^kR_M$ for each $k>0$. Thus  for each $k$ either $QR_M \subseteq J^kR_M$  or $J^kR_M \subseteq QR_M$.
If there is $k>0$ such that $J^kR_M \subseteq QR_M$, then 
the fact that $M = \sqrt{J}$ implies that $M = Q$.  Otherwise, if $QR_M \subseteq J^kR_M$ for all $k>0$, then since  $J$ is $M$-primary we conclude that $Q \subseteq J^k$ for all $k>0$. Thus $Q 
 \subseteq \bigcap_{k>0}J^k = P$, which 
 shows that every prime ideal of $R$ properly contained in $M$ is contained in $P$.  

(2)
Let $M_1,\ldots,M_n$ denote the maximal ideals of $R$. As in the proof of (1), there is for each $i$ a principal ideal $J_i$ such that $M_i = \sqrt{J_i}$.  
For each $i$,  let $P_i = \bigcap_{k>0}J_i^k$. As we have shown, each $P_i$ is the unique largest nonmaximal prime ideal of $R$ contained in $M_i$.   
 Using the comaximality of the $J_i$, we have  $C: = \bigcap_{k>0}m^kR = \bigcap_{k>0} (J_1^k \cap \cdots \cap J^k_n) = P_1 \cap \cdots \cap P_n$.  Therefore, $R^* = R/C$ is a reduced ring with finitely many minimal prime ideals.

   (3) 
 To see that   $R^*$ is  Pr\"ufer 
   in dimension zero, 
 let $I$ be a zero-dimensional  ideal of
  $R$ containing $C$ such that $I/C$ is a finitely generated ideal of $R^*$.  By Lemma~\ref{n generator}, $I$ is a finitely generated ideal of $R$, and hence  invertible by assumption. Since $R$ is semilocal, $I$ is then principal, and hence $I/C$ is a principal ideal of $R^*$. Since also $I/C$ is zero-dimensional, $I/C$ is a regular ideal of $R^*$ (it contains a power of the nonzerodivisor $m+C$) and hence $I/C$ is invertible. Thus $R^*$ is Pr\"ufer in dimension zero.
 
 %
%

(4)
After relabeling, we may assume there exists $t \leq n$ such that $P_1, \ldots, P_t$ are the minimal primes  over $C$.  For each $i \leq t$, if $Q $ is a prime ideal properly contained in $M_i$ and containing $C$, then, as we have established in (1), $C \subseteq Q \subseteq P_i$, so that since $P_i$ is minimal over $C$, we have $Q = P_i$. Thus for $i \leq t$,   $M_i$ is height one prime ideal of $R^*$. 

(5)
Let  $M$ be  a maximal ideal of $R$ such that $M/C$ has height one in $R^*$. Each minimal prime ideal of $R^*$ is of the form $P_i/C$ for some $i \leq t$.  
Suppose  $P_i + P_j \subseteq M$ for some  $i,j \leq t$.  Let  $k \leq n$ such that $M = M_k$.  Since $P_k $ is the unique largest 
nonmaximal prime ideal  of $R$ contained in $M_k$, we have $P_i +P_j \subseteq P_k \subseteq M_k$.  Since $C \subseteq P_i \cap P_j$ and $M_k/C$ has height one, we have $P_i = P_k  = P_j$. This shows that  $M/C$ contains a unique minimal prime ideal of $R^*$. Since $M/C$ has height one, it follows that $M/C$ contains a unique nonmaximal prime ideal of $R^*$. 
 \end{proof}

\begin{corollary} \label{Prufer subring local}  Let $R$ be a  local  ring whose maximal ideal $M$ is the radical of a regular principal ideal $I$. Then $R$  
is  Pr\"ufer in dimension zero if  and only if $R^*=R/(\bigcap_{k>0}I^k)$ is a  valuation domain of Krull dimension one. 
\end{corollary}

\begin{proof} Suppose  $R$ is  Pr\"ufer in dimension zero. By Theorem~\ref{new lemma},  $R^*$ is a reduced ring that is Pr\"ufer in dimension zero. This theorem also shows that since $R^*$ is local, there is a unique minimal prime ideal of $R^*$ and the maximal ideal of $R^*$ has height one. Therefore, as a one-dimensional local Pr\"ufer domain, $R^*$ is a one-dimensional valuation domain.

Conversely, suppose $R^*$ is a one-dimensional valuation domain.
  Let $I$ be a finitely generated zero-dimensional ideal of $R$.  Then the image 
$I^*$ of $I$ in $R^*$ is a finitely generated  ideal. Since $R^*$ is a valuation domain,  $I^*$ is a principal ideal in $R^*$. Therefore, $I$ is a principal ideal in $R$ by Lemma~\ref{n generator}. 
Since $M$ contains a nonzerodivisor and  $I$ is a zero-dimensional ideal, it follows that $I$ is  a regular ideal, and hence $I$ is invertible. Thus $R$ is Pr\"ufer in dimension zero. 
\end{proof}

We show in Corollary~\ref{int closed} that an integrally closed local ring with finitely generated regular maximal ideal $M$ has the property that $M $ is the radical of a principal ideal if and only if $M$ is a principal ideal. Such a ring clearly is Pr\"ufer in dimension zero since every zero-dimensional ideal is principal. 
The next example shows that without the assumption of finite generation of the maximal ideal, an integrally closed local domain whose maximal ideal is the radical of a principal ideal need not even be Pr\"ufer in dimension zero.

\begin{example} {An integrally closed local domain $R$ whose maximal ideal $M$ is the radical of a principal ideal $I$ but for which $R$ is not Pr\"ufer in dimension zero.} 
%
 {\em This example is based on \cite[Example 5.3]{Ohm}.  Let $p$ be  a prime integer, let $v_p$ be the $p$-adic valuation on ${\mathbb{Q}}$, and let $v$ be the rank one valuation  defined on ${\mathbb{Q}}[X]$ by 
 $$v(a_0 + a_1X+\cdots + a_nX^n) = \inf \{v_p(a_i)+i \pi:i=0,1,\ldots,n\},$$ where $a_0,\ldots,a_n \in {\mathbb{Q}}$, and extended to ${\mathbb{Q}}(X)$ by defining $v(f/g) = v(f)  - v(g)$ if $f,g \in {\mathbb{Q}}[X]$ with $g \ne 0$. 
   Let $V$ be the valuation ring in ${\mathbb{Q}}(X)$ corresponding to $v$, and define $A = V \cap {\mathbb{Q}}[X]$. Finally, let $R = A_{{{\mathfrak M}_V} \cap A}$, where ${\mathfrak{M}}_V$ is the maximal ideal of $V$, and let $M$ denote the maximal ideal of $R$.  Then $M = \sqrt{pR}$ and $M$ has height $2$ \cite[Lemmas 5.4 and 5.8]{Ohm}.   Yet since $p \in {\mathfrak{M}}_V$ and $V$ is  valuation ring with Krull dimension one, we have $\bigcap_{i}p^iR \subseteq \bigcap_i p^iV = 0$. 
In the notation of Corollary~\ref{Prufer subring local}, $R^* = R$.     
   Since $R$ is not a valuation ring, Corollary~\ref{Prufer subring local} implies $R$ is not Pr\"ufer in dimension zero. }
\end{example}

\section{Principal ideal theorems}  


\label{Noetherian}

In this section we prove principal ideal theorems for the case in which 
$R$ is a local ring  with finitely generated maximal ideal $M$. With the exception of Theorem~\ref{pre pre stable}, the focus is on the case in which $M$ is also regular. As discussed in the introduction, this is a case to which it is always possible to reduce.   

The proof of our first version of a principal ideal theorem, Theorem~\ref{pre pre stable}, is mainly a matter of applying the work of Eakin and Sathaye \cite{ES} and Sally \cite{Sbook}.   For a local ring $R$ with maximal ideal $M$, we denote by $\widehat{R}$ the $M$-adic completion of $R$.

\begin{theorem} \label{pre pre stable} The following are equivalent for a local ring $R$ with finitely generated maximal ideal $M$. 
\begin{itemize}

\item[(1)] $M$ is the radical of a principal ideal of $R$. 


\item[(2)] $\widehat{R}$ is a  (Noetherian)  ring of Krull dimension at most one.

\item[(3)] There is $n>0$ such that every $M$-primary ideal of $R$ can be generated by $n$ elements.

\item[(4)] For each $M$-primary ideal $I$ of $R$ there is $n>0$ and $x \in I^n$ such that $I^{2n} = xI^n$.

\end{itemize}
\end{theorem}

\begin{proof}
(1) $\Rightarrow$ (2) 
Let $m \in M$ be such that $M = \sqrt{mR}$. 
Since the maximal ideal of $R$ is finitely generated, $\widehat{R}$ is a Noetherian ring \cite[Theorem 3]{Cohen} whose maximal ideal is the radical of $m\widehat{R}$. By Krull's PIT,  $\widehat{R}$ has Krull dimension at most one. 

(2) $\Rightarrow$ (3) 
By \cite[Theorem 1.2, p.~51]{Sbook}, the fact that $\widehat{R}$ is a local Noetherian ring of Krull dimension at most one implies  there is $n>0$ such that every ideal of $\widehat{R}$ can be generated by $n$ elements.
Let $I$ be an $M$-primary ideal of $R$. Since $M$ is a finitely generated ideal, $I$ is a finitely generated ideal and  there is   $k>0$ such that $M^k \subseteq I$.   Now $I\widehat{R}/M^{k+1}\widehat{R} \cong I/M^{k+1}$ and $\widehat{R}/M^{k+1}\widehat{R} \cong R/M^{k+1}$. Since $I\widehat{R}$ can be generated by $n$ elements as an $\widehat{R}$-module, 
 we conclude that $I/M^{k+1}$ can be generated by $n$ elements as an $R$-module. Since $M^{k+1} \subseteq IM$, Nakayama's Lemma implies that $I$ can be generated by $n$ elements. 
 
 (3) $\Rightarrow$ (4) Let $I$ be an $M$-primary ideal of $R$. By (3), $I^n$ can be generated by $n$ elements. This  implies that $I^{2n} = yI^n$ for some $y \in I^n$ \cite[Corollary~1, p.~446]{ES}.  
 
 (4) $\Rightarrow$ (1) By (4), there exists $n>0$ and $x \in M^n$ such that $M^{2n} = xM^n$. Thus $M = \sqrt{xR}$ since $M^{2n} \subseteq xR$.  
\end{proof}

The next lemma, which in its proof makes  use of the existence of a uniform bound due to Rees, is applied in Lemma~\ref{lift} to prove a lifting property for completions. In this sense Lemma~\ref{bound} can be viewed as a substitute for the Artin-Rees lemma in our non-Noetherian context. 
We denote by $\overline{I}$ the integral closure of the ideal $I$ in the ring $R$, that is, $\overline{I}$ is the set of all $r \in R$ for which there is  $n>0$ and $a_k \in I^k$ ($k=1,2,\ldots,n$) such that $$r^n + a_{1}r^{n-1} + \cdots + a_{n-1}r + a_n = 0.$$

\begin{lemma} \label{bound} If $R$ is a local ring with finitely generated maximal ideal $M$, then for each  $M$-primary ideal $I$ of $R$ there is an integer $k>0$ such that  $(\overline{I^n})^{k} \subseteq I^n$ for every $n > 0$.  
\end{lemma}

\begin{proof}  By \cite[Theorem 3]{Cohen}, $\widehat{R}$ is a Noetherian ring. 
 Let $I$ be an $M$-primary ideal of $R$, let $N$ be the nilradical of $\widehat{R}$, let $S  = \widehat{R}/N$ and let $J = IS$. 
   By a theorem of Rees (see \cite[Theorem 9.1.2]{SH}), since $S=\widehat{R}/N$ is reduced,   there is an integer $\ell\geq 2$ such that the integral closure  $\overline{J}$ of $J$ in $S$ has the property that for all $n\geq 0$,  $\overline{J^{n+\ell}} \subseteq J^n$.  
Now for each $n\geq 2$, since $\ell \geq 2$, we have $n \geq 2 \geq \ell/(\ell-1)$.  Thus $n(\ell-1) \geq \ell$, so that $n\ell \geq \ell + n$.  Consequently,  $\overline{J^{n \ell}} \subseteq \overline{J^{\ell +n}} \subseteq J^n.$
 Since $\overline{J^n}^{\:\ell} \subseteq \overline{J^{n\ell}}$ \cite[Remark 1.3.2(4)]{SH}, it follows that $(\overline{J^n})^{\ell} \subseteq J^n$.  
Now  $\overline{I^n}S \subseteq \overline{I^nS} = \overline{J^n}$ \cite[Remark 1.1.3(7)]{SH}, so $$(\overline{I^n})^{\ell} S \subseteq (\overline{J^n})^{\ell} \subseteq J^n = I^nS.$$
Hence $(\overline{I^n})^{\ell} \widehat{R} \subseteq I^n\widehat{R} + N$. Let $t>0$ be such that $N^t = 0$. Then $$(\overline{I^n})^{ \ell  t}\widehat{R} \subseteq (I^n\widehat{R} + N)^t \subseteq I^n\widehat{R}.$$  
Since $I$ is $M$-primary and $M$ is finitely generated, there is $u>0$ such that $M^u \subseteq (\overline{I^n})^{ \ell  t} \cap I^n.$ Since $(\overline{I^n})^{ \ell  t}\widehat{R} \subseteq  I^n\widehat{R}$, it follows that  $(\overline{I^n})^{ \ell  t}/M^u \subseteq I^n/M^u$. Therefore, 
for each $n \geq 2$,  we have $(\overline{I^n})^{ \ell  t} \subseteq  I^n$.  
Since $\ell$ and $t$ depend only on $I$ and the ring $R$, respectively, the lemma is proved for all $n\geq 2$. Finally, in the case $n =1$, since $I$ and  ${\overline{I}}$ have the same radical, we have that $(\overline{I})^{ s} \subseteq I$ for some $s>0$. Thus we obtain the lemma for all $n\geq 1$ by choosing $k = \max \{\ell t,s\}$.   
\end{proof}


\begin{lemma} \label{lift} Let $R$ be a local ring with finitely generated maximal ideal $M$, and let $S$ be a finite  extension of $R$ in $Q(R)$.  Then the inclusion mapping $R \rightarrow S$ lifts to an embedding of rings $\widehat{R} \rightarrow \widehat{S}$, where both completions are with respect to the $M$-adic topology. Moreover, $\widehat{S}$ is a finite extension of $\widehat{R}$. 
\end{lemma}  

\begin{proof} We claim that the $M$-adic topology on $R$ is the  subspace topology on $R$ induced by the $MS$-adic topology on $S$. Since $S$ is integral over $R$ with $S \subseteq Q(R)$, we have for each $n>0$ that $\overline{M^n} = M^nS \cap R$.  Thus the filtration $\{\overline{M^n}:n >0\}$ is a basis of open neighborhoods of $0$ for the subspace topology on $R$.  
 By Lemma~\ref{bound} there exists $k>0$ such that $\overline{M^n}^k \subseteq M^n$ for every $n>0$, so it follows that $\{\overline{M^n}:n >0\}$ is a basis of open neighborhoods of $0$ for the $M$-adic topology. This shows that the $M$-adic topology on $R$ is the  subspace topology on $R$ induced by the $MS$-adic topology on $S$. 
Hence the canonical mapping $\widehat{R} \rightarrow \widehat{S}$ is an embedding of rings \cite[Theorem 8.1, p.~57]{Ma}. Since $ \widehat{S}/M\widehat{S} = S/MS $ is a finite extension of $\widehat{R}/M\widehat{R} = R/MS$ and $\widehat{R}$ is complete, we have  that $\widehat{S}$ is a finite extension of $\widehat{R}$ \cite[Theorem 8.4, p.~58]{Ma}.  
\end{proof} 

The Krull-Akizuki Theorem states that every integral extension of a one-dimen\-sion\-al Noetherian domain inside a finite extension of its quotient field is a Noetherian domain.  
Using an approach to the Krull-Akizuki Theorem due to Matijevic \cite{Mati}, we 
 next give a version of this theorem that is suited to our more general context. 
For an ideal $I$ of a ring $R$, we denote by $T(I)$ the {\it Nagata transform} of $I$; that is, $T(I)  = \bigcup_{k>0} (R:_{Q(R)}I^k)$.

\begin{lemma} \label{Mati} Let  $R$ be a local ring with finitely generated regular maximal ideal $M$ that is the radical of a principal ideal. If  $S$ is an integral extension  of $R$ in $ T(M)$,  then $S$ is semilocal and the maximal ideals of $S$ are finitely generated. 
\end{lemma}

\begin{proof}  The proof is based on an argument due to Matijevic \cite[Theorem]{Mati}. Let $m$ be an element of $R$ such that $M = \sqrt{mR}$.   
 We show first that there is $n>0$ such that 
  $S \subseteq Rm^{-n} + mS$. 
For each $t>0$, let $I_t = (m^tS \cap R) + mR$.   Since $R/mR$ is an Artinian ring, there is $n>0$ such that $I_n = I_{n+k}$ for all $k>0$.  We show that $S \subseteq m^{-n}R +mS$ for this $n$. 
 Suppose to the contrary there is $s \in S$ such that $s \not \in m^{-n}R+mS$. 
  Since $S \subseteq R[1/m]$, there is a minimal choice of $k>n$ such that $s \in m^{-k}R+mS$. Let $r \in R$ and $x \in S$ such that $m^ks = r + m^{k+1}x$. Then $m^k(s-mx) \in m^kS \cap R \subseteq I_k$. Since $k>n$, we have by the choice of $n$ that $I_k = I_{k+1}$, so that $m^k(s-mx) \in (m^{k+1}S \cap R) + mR$.
    Let $u \in S$ and $a \in R$ such that $m^k(s-mx) = m^{k+1}u+ma$. Then $m^ks = m^{k+1}(x+u) + ma$. Since $m$ is a nonzerodivisor in $S$, $s \in mS + m^{-(k-1)}R$, contrary to the choice of $k$. 
  Therefore, $S \subseteq Rm^{-n}+mS$.

  Now $(Rm^{-n}+mS)/mS$ is  finitely generated as a module over the Noetherian ring $R/mR$, so the $R/mR$-submodule $S/mS$ is also  finitely generated, and hence $S/mS$ is a Noetherian ring.  Since $S$ is integral over $R$, every maximal ideal $N$ of $S$ contains $mS$. Since also $N/mS$ is a finitely generated ideal of $S/mS$, we conclude that $N$ 
   is a finitely generated ideal of $S$.  Since $S$ is integral over $R$,  $mS$ is a zero-dimensional ideal of $S$, and hence $S/mS$ is a zero-dimensional Noetherian, hence semilocal, ring. This proves $S$ is a semilocal ring with finitely generated maximal ideals.  
\end{proof}

The next theorem is our main principal ideal theorem for the case in which the maximal ideal is regular.



\begin{theorem} \label{pre stable} Let $R$ be a  local ring with finitely generated regular maximal ideal $M$,   let $R'$ be the integral closure of $R$ in $T(M)$ and let  $\overline{R}$ be the integral closure of $R$ in $Q(R)$.  Then the following are equivalent. 
\begin{itemize}

\item[(1)] $M$ is the radical of a principal  ideal of $R$. 


\item[(2)] $\widehat{R}$ is a one-dimensional Cohen-Macaulay ring. 

\item[(3)] $R'$  is Pr\"ufer in dimension zero.

\item[(4)] There is an integral overring of $R$ that is Pr\"ufer in dimension zero.




\item[(5)] Every $M$-primary ideal of $R$ extends to a principal ideal of $\overline{R}$.


\item[(6)] $\Jac \overline{R}$   is the radical of a principal ideal of $\overline{R}$.


\item[(7)] $R'$ is a semilocal ring in which every maximal ideal is  principal.

\item[(8)]  There is an integral overring of $R$ in which  every maximal ideal is  principal.

\item[(9)] $T(M)$ is a semilocal ring for which $MT(M) = T(M)$.



\end{itemize}
\end{theorem}

\begin{proof}
(1) $\Rightarrow$ (2) 
By (1), there exists a nonzerodivisor $m \in M$ such that $M = \sqrt{mR}$. 
By Theorem~\ref{pre pre stable},  $\widehat{R}$ is a Noetherian ring of Krull dimension at most one. Also, the maximal ideal  $M\widehat{R}$ of $\widehat{R}$ is the radical of the principal ideal   ${m\widehat{R}}$.   To see that the image of $m$ in $\widehat{R}$ is a nonzerodivisor, suppose that $x \in \widehat{R}$ with $xm = 0$.  Let $\{x_i\}_{i=1}^\infty$ be a Cauchy sequence of elements in $R$ whose image in $\widehat{R}$ converges to $x$. Since $M$ is a finitely generated ideal and $M  = \sqrt{mR}$, 
 the $mR$-adic and $M$-adic topologies agree on $R$.  Thus we can assume that 
 $x_{i+1}-x_i \in m^iR$ for all $i>0$. Also, since $xm = 0$, we have that 
 for each $n>0$, there exists $i_n>n$ such that $x_{i_n}m \in m^nR$.  Since $m$ is a nonzerodivisor in $R$, we have $x_{i_n} \in m^{n-1}R$. Since also $x_{i_n} - x_{i_{n-1}} \in m^{i_{n-1}} R \subseteq m^{n-1}R$, we have then that $x_{i_n -1} \in m^{n-1}R$.  Thus $\{x_{i_n}\}_{n=1}^\infty$ is a Cauchy subsequence of $\{x_i\}_{i=1}^\infty$ whose image in $\widehat{R}$ converges to $0$. Therefore, $x = 0$,   
which proves that   the image of $m$ in $\widehat{R}$ is a nonzerodivisor.  We conclude that 
 $\widehat{R}$  is a one-dimensional Cohen-Macaulay ring.

 (2) $\Rightarrow$ (3)  
By Theorem~\ref{pre pre stable}, there is  $m \in M$ such that $M = \sqrt{mR}$.  
 Since $M$ is a regular ideal, $m$ is a nonzerodivisor of $R$ and  $T(M) = R[1/m]$.  To show that $R'$ is Pr\"ufer in dimension zero, 
%
 let $I$ be a finitely generated zero-dimensional ideal of $R'$, say $I = (x_1,\ldots,x_t)R'$. We show $I$ is invertible. Assume $I \ne R'$.  
   Since $R'$ is an integral extension of $R$, $\Jac R' = \sqrt{mR'}$, and hence  $m^k \in I$ for some $k>0$. Write $m^k = \sum_{i=1}^t \alpha_i x_i$, where $\alpha_i \in R'$.  Since $I \subseteq R' \subseteq R[1/m]$, 
    there is $\ell>0$ such that $m^\ell x_i, m^\ell \alpha_i \in R$ for all $i$. Let $J=m^\ell (x_1,\ldots,x_n)R$. Then $J$ is an ideal of $R$ and  $$m^{k+\ell} = m^\ell(\sum_{i=1}^t \alpha_i x_i) = \sum_{i=1}^t (m^\ell \alpha_i)x_i \in (x_1,\ldots,x_t)R.$$ Thus $m^{k+2\ell} \in (m^\ell x_1,\ldots, m^\ell x_t)R \subseteq R$. 
      If $J = R$, then $1 \in JR' =  m^\ell I \subseteq m^\ell R'$, contrary to the fact that $m \in \Jac R'$. Thus $J$ is a proper ideal of $R$.   
    Since also $m^{k+2\ell} \in J$, we have that $J$ is an $M$-primary ideal of $R$.
   
    By Theorem~\ref{pre pre stable},
        there is $n>0$ such that $J^{2n} = xJ^n$ for some $x \in J^n$. 
   By \cite[Lemma, p.~447]{ES}, $J^n$ is a principal ideal of $\End(J^n):=(J^n:_{Q(R)}J^n)$. Since $J^n$ is a finitely generated ideal of $R$, $\End(J^n)$ is integral over $R$  \cite[Lemma 2.1.8]{SH}. 
   But $\End(J^n) \subseteq R[1/ m]$ since $JR[1/m]=R[1/m]$, so necessarily $\End(J^n) \subseteq R'$. Thus since $J^n$ is a principal ideal of $\End(J^n)$, we have that $J^nR'$ is a principal ideal of $R'$.  Now $J^nR' = m^{\ell n} I^n$, so since $m$ is a nonzerodivisor and $J^n$ is  a principal ideal of $R'$ it must be that $I^n$ is a principal ideal of $R'$.
  Since some power of $m$ is in $I^n$, the ideal $I^n$ is regular.  
    Hence $I^n$ is an invertible ideal of $R'$, and so  $I$ is also an invertible ideal of $R'$.  
   This shows that $R'$ is Pr\"ufer in dimension zero. 
   
   (3) $\Rightarrow$ (4) This is clear. 
   
   (4) $\Rightarrow$ (5) Let $S$ be an integral overring of $R$ that is Pr\"ufer in dimension zero. Let $I$ be an $M$-primary ideal of $R$.  Since $M$ is finitely generated, so is $I$. As a finitely generated zero-dimensional ideal of $R$, $IS$ is a finitely generated regular, hence invertible, ideal of $S$.  Thus $1 \in (S:_{Q(R)}I)I$.  Write $I = (x_1,\ldots,x_n)R$.  Then there exist $q_1,\ldots,q_n \in (S:_{Q(R)}I)$ such that $1 = \sum_i q_ix_i$.  Let $S' = R[q_ix_j:i,j \leq n]$.  Then $S' \subseteq S$ and each $q_i \in (S':_{Q(R)}I)$.  Thus $1 \in (S':_{Q(R)}I)I$, so that $IS'$ is an invertible ideal of $S'$.  Since $S'$ is a module-finite extension of $R$, $S'$ is a semilocal ring. Thus $IS'$ is a principal ideal of $S'$, which implies $IS$ is a principal ideal of $S$.  Since $S \subseteq \overline{R}$, statement (5) now follows.

   (5) $\Rightarrow$ (6) By (5), $M\overline{R}$ is a principal ideal of $\overline{R}$. Since $\overline{R}$ is integral over $R$, we have $\Jac \overline{R} = \sqrt{M\overline{R}}$, which verifies (6).  
   
     
     (6) $\Rightarrow$ (1) Let $x \in \overline{R}$ such that  $\Jac \overline{R}$ is the radical of  ${x\overline{R}}$, and let $S = R[x]$.  Then $S$ is a module-finite extension of $R$ with $\Jac S = \sqrt{xS}$.     
   By Lemma~\ref{lift} the canonical homomorphism $\widehat{R} \rightarrow \widehat{S}$ is an embedding and $\widehat{S}$ is integral over the image of $\widehat{R}$ in $\widehat{S}$.  Since    
    $\Jac \widehat{S} = (\Jac S)\widehat{S}$, it follows that $\Jac \widehat{S}$ is the radical of a principal ideal of $\widehat{S}$.  As a finite extension of the Noetherian ring $\widehat{R}$,  the ring $\widehat{S}$ is  Noetherian. Therefore, Krull's PIT implies that $\dim \widehat{S} \leq 1$.  Since $\widehat{S}$ is integral over $\widehat{R}$, we conclude that $\dim \widehat{R} \leq  1$.  By Theorem~\ref{pre pre stable} the maximal ideal of $R$ is the radical of a principal ideal of $R$.

    (3) $\Rightarrow$ (7)  
    Assume (3). We have shown that (3) is equivalent to (1), so $M$ is the radical of a principal ideal of $R$.  By Lemma~\ref{Mati}, $R'$ is semilocal and the maximal ideals of $R'$ are finitely generated. Since $R'$ is Pr\"ufer in dimension zero, the maximal ideals of $R'$ are invertible. Since $R'$ is semilocal, these maximal ideals are  principal. 
    
    (7) $\Rightarrow$ (8) This is clear.

      (8) $\Rightarrow$ (9)  Let $S$ be an integral overring of $R$ for which the maximal ideals of $S$ are principal. Then $S/(\Jac R)S$ is a zero-dimensional ring with principal maximal ideals, and hence $S/(\Jac R)S$ is Artinian.  Since $S$ is integral over $R$, we have $\Jac R \subseteq \Jac S$,  so that  $S/\Jac S$ is Artinian. In particular,  
        $S$ is semilocal. Again using the fact that the maximal ideals of $S$ are principal, 
    we have that    $S$ is Pr\"ufer in dimension zero (indeed, every zero-dimensional ideal is a regular principal ideal). By Theorem~\ref{new lemma}(1),       each maximal ideal 
     of $S$ contains a unique largest nonmaximal prime ideal, so there
      are finitely many prime ideals $P_1,\ldots,P_k$ that are maximal among the set of nonmaximal prime ideals of $S$, and each nonmaximal prime ideal is contained in one of the $P_i$. 
     
      Since $S$ is an integral extension of $R$, 
   each nonmaximal prime ideal $P$ of $R$ is contained in $P_i \cap R$ for some $i$. Also since $S$ is integral over $R$, no $P_i$ lies over $M$. By Prime Avoidance, there is $m \in M \setminus (P_1 \cup \cdots \cup P_k)$. Thus $M = \sqrt{mR}$ and, since $M$ is regular, $m$ is a nonzerodivsor of $R$. Thus $T(M) = R[1/m]$ and the maximal ideals of $T(M)$ are all of the form $(P_i \cap R)T(M)$. This shows that $T(M)$ is semilocal. Since $T(M) = R[1/m]$, we have $MT(M) = T(M)$, which verifies (9).

(9) $\Rightarrow$ (1)  Let $P_1,\ldots,P_n$ denote the maximal ideals of $T(M)$.  Since $MT(M) = T(M)$,  none of the $P_i$ lie over $M$, so Prime Avoidance implies there exists $m \in M \setminus (P_1 \cup \cdots \cup P_n)$.  Then $m$ is a unit in $T(M)$. Let $k>0$ such that $1/m \in (R:_{Q(R)}M^k)$. Then $M^k \subseteq mR \subseteq M$, so $M = \sqrt{mR}$, which verifies (1). 
   \end{proof}


\begin{corollary} \label{int closed} Let $R$ be a local ring with finitely generated regular  maximal ideal $M$ and such that $R$ is integrally closed in $Q(R)$.   Then $M$ is the radical of a principal ideal if and only if $M$ is a principal ideal. 
\end{corollary}

\begin{proof} 
If $M$ is the radical of a principal ideal, then, since $R$ is integrally closed, Theorem~\ref{pre stable} implies that $M$ is a principal ideal.
The converse is clear. 
\end{proof}



\section{Generalized local rings}

In this section we consider principal ideal theorems for generalized local rings.  
The motivation for this case is  the last reduction discussed in the introduction, that of the fact that it always possible to pass from a local ring  $R$ with a finitely generated regular maximal ideal which is the radical of a principal ideal to  a generalized local ring $A$ with this same property. This is done by showing in Theorem~\ref{glr theorem} that every such ring $R$ arises  as a pullback  involving a generalized local ring $A$ and a semilocal ring $T$.

\begin{theorem} \label{glr theorem} Let $R$ be a local ring. The maximal ideal of $R$ is  finitely generated, regular and the radical of a  principal ideal if and only if $R$ occurs in a pullback diagram of the form

\begin{center} $\begin{CD} R = \beta^{-1}(A) @>>> A \\ @VVV \: @VV{\iota}V \\
   T @>{\beta}>> B \end{CD}$ \end{center}
where $A$, $B$ and $T$ are rings such that 
\begin{itemize}
\item[(a)] $\beta:T \rightarrow B$ is a surjective ring homomorphism,  $A$ is a subring of $B$ and  $\iota:A \rightarrow B$ is the inclusion map,
\item[(b)] $A$ is a generalized local ring with regular maximal ideal that is the radical of a principal ideal,
\item[(c)] $B$ is the Nagata transform of the maximal ideal of $A$, and
\item[(d)] $T$ is a (necessarily semilocal) ring.

\end{itemize}
\end{theorem}

\begin{proof}
Suppose the maximal ideal $M$ of $R$ is finitely generated, regular and  the radical of a  principal ideal, say $M = \sqrt{mR}$. Then $m$ is a nonzerodivisor in $R$.    Let
$C = \bigcap_{k>0}m^kR$, $A = R/C$ and  $T = T(M)$. Since $M$ is finitely generated and $M  = \sqrt{mR}$, we have $T = R[1/m]$. It follows that $C$ is an ideal of $T$. Let $B = T/C$.  Then $R$ occurs in a pullback diagram as in the statement of the theorem and  (a) is clear. It remains to verify (b), (c) and (d).  

Since $M$ is a finitely generated ideal and $M = \sqrt{mR}$, it follows that $C = \bigcap_{i>0}M^i$. Thus 
 $A$ is a generalized local ring and  $M/C = \sqrt{mR/C}$. To see that $m+C$ is a nonzerodivisor in $A  = R/C$, suppose $r \in R$ such that $rm \in C$. For each $i>1$, $rm \in m^{i+1}R$, so that since $m$ is a nonzerodivisor in $R$, $r \in m^iR$. Thus $r \in C$, which shows that $m+C$ is a nonzerodivisor in $A$. This verifies (b). 
Next, since $M/C = \sqrt{mR/C}$ and $m+C$ is a nonzerodivisor in $R$, we have $T(M/C) = R[1/m]/C = T/C = B$.  Thus (c) is satisfied by our choice of $B$.
That $T$ is semilocal follows from Theorem~\ref{pre stable}(9). Thus (d) holds. 


Conversely, suppose $R$ occurs in a pullback diagram as in the theorem, where (a)--(d) are satisfied. 
 Let $C$ be the kernel of the map $T \rightarrow B$, and observe that $C \subseteq R$.  Without loss of generality we may assume $A = R/C$ and $B = T/C$. Let $m \in M$ such that $M/C$ is the radical of $(mR+C)/C$. 

We claim first that $mT = T$.  Since  $M/C = \sqrt{(mR+C)/C}$ and $M/C$ is a regular ideal, it follows that  $m+C$ is a nonzerodivisor in $A$. Since $B$ is the Nagata transform of $M/C$, we have also that 
 $mB = B$. Thus  $mT + C = T$, and there exist $t \in T$ and $c \in C$ such that $1 = mt + c$. Consequently, $mt = 1 - c \in R$.  Since $c \in M$ and $R$ is local,  $mt$ is a unit in $R$, hence a unit in $T$.  Therefore, $mT = T$. 

Next we claim that $C \subseteq mR$.  Let $c \in C$.  Since $mT = T$, there is $t \in T$ such that $1 = mt$.  Since $C$ is an ideal of $T$ that is contained in $R$, we have  $c = m(tc) \in mR$ for all $c \in C$. Thus $C \subseteq mR$.

Now since $M/C$ is a finitely generated ideal that is the radical of $(mR+C)/C = mR/C$, we conclude there is $k>0$ such that $M^k \subseteq mR$.  Consequently, $M = \sqrt{mR}$.  Since $M/C$ is finitely generated,  there exist $x_1,\ldots,x_n \in M$ such that $M =  (x_1,\ldots,x_n)R + C$. Since $C \subseteq mR$, we have $M = (x_1,\ldots,x_n,m)R$, proving that $M$ is finitely generated.   Finally, since $mT = T$, it follows that $m$ is a nonzerodivisor of $R$. This proves that $M$ is a finitely generated regular ideal that is the radical of a principal ideal. 
\end{proof}

For the rest of the section we focus on generalized local rings that have a maximal ideal that is the radical of a principal ideal. The following examples show that such  rings may have Krull dimension more than one.  

\begin{example} \label{new examples} $\:$
{\em \begin{itemize}

\item[(1)]  (Cahen, Houston and Lucas \cite{CHL}) Let $A$ be a subset of the $p$-adic integers $\widehat{\mathbb{Z}}_p$ that contains $p\widehat{{\mathbb{Z}}}_p$ and elements $\alpha,\beta$ such that the image of $\alpha$ mod $p$ is algebraic over ${\mathbb{Z}}/p{\mathbb{Z}}$ while the image of $\beta$ is transcendental. 
Let  $S$ be the subring of $\widehat{{\mathbb{Q}}}_p[X] $ defined by $S= \{f(x) \in \widehat{\mathbb{Q}}_p[X]:f(A) \subseteq \widehat{\mathbb{Z}}_p\}.$ For each $\gamma \in A$, let ${\ff M}_\gamma$ be the maximal ideal of $S$ given by  ${\ff M}_\gamma = \{f \in S:f(A) \subseteq {\mathbb{Z}}_p\}$.  
Then the ring  $R= {\mathbb{Z}} + ({\ff M}_\alpha \cap {\ff M}_{\beta})$ is a generalized  local ring of Krull dimension 2 with  maximal ideal $M = {\ff M}_\alpha \cap {\ff M}_\beta$ that is minimal over a principal ideal  of $R$. Moreover, the set of principal ideals of $R$ satisfies the ascending chain condition; see the discussion following Proposition 3 of  \cite{CHL}.

\item[(2)]  (Gabelli and Roitman \cite{GR}) It is shown in \cite[Corollary 4.8]{GR} that there exists for each $n>0$ a generalized local domain $R$ of Krull dimension $n$ such that for each ideal $I$ of $R$, there is $x \in I$ such that $I^2 = xI$.  As in (1), this ring has the ascending chain condition on principal ideals. This ring is constructed as a pullback of an intersection of a DVR and a rank $n$ discrete valuation domain. The example in (1) is also such a pullback. Theorem~\ref{at least one} sheds additional light on why in each of these examples one of these valuation rings must have rank one. 

\item[(3)] (Heinzer, Rotthaus and Wiegand \cite{HRW}) 
 Let $k$ be a field, and let $X,Y$ be indeterminates for $k$. In \cite[Chapter~17]{HRW} a non-Noetherian two-dimensional  local domain $S$ is constructed with the properties that the maximal ideal  $N$ of $S$  is a finitely generated ideal  (it can be generated by two elements), $\bigcap_{i}N^i = 0$ and $S$ has  only 3 prime ideals. (What we term ``$S$'' here is the ring  $B/q$ in  
 the discussion of Type II prime ideals following Discussion 17.5 in \cite{HRW}.)
Moreover, $S$ is  the quotient of a ring  between $k[X,Y]$ and $k[[X,Y]]$ that is  dominated by $k[[X,Y]]$, and hence 
 $S$ is a $k$-algebra with residue field $k$. As such, $S = k + N$.  Since $S$ has only finitely many prime ideals it follows that $N$ is the radical of a principal ideal of $S$. 

\end{itemize}}
\end{example}

The examples show that 
Krull's PIT cannot be extended from the class of Noetherian local rings to that of all  generalized local rings.  In Theorem~\ref{count} we give a special case in which the PIT holds  under an additional assumption on the cardinalities  of the set  of   minimal prime ideals of  $R$ and $\widehat{R}$.

\begin{lemma} \label{contraction} If $R$ is a    generalized local ring whose maximal ideal is regular and the radical of a principal ideal, then  each  minimal prime ideal of $R$ and each   one-dimensional prime ideal of ${R}$    is contracted from a minimal prime ideal of $\widehat{R}$.  
\end{lemma}

\begin{proof} 
Since $R$ is a generalized local ring, $R$ embeds into $\widehat{R}$.   Thus
 every minimal prime ideal of $R$ is contracted from a minimal prime ideal of $\widehat{R}$. 
 Next, suppose that $P$ is a one-dimensional prime ideal of $R$.  
  We claim    $P\widehat{R}$ is  a height $0$ ideal of $\widehat{R}$. Suppose this is not the case. 
  Let $M$ denote the maximal ideal of $R$, and 
 let $m \in M$ such that $M = \sqrt{mR}$.  By Theorem~\ref{pre stable}, 
  $\widehat{R}$ is a local Noetherian ring of Krull dimension $1$. The fact that $P\widehat{R}$ has height greater than $0$ implies that $P\widehat{R}$ is $M\widehat{R}$-primary. 
Thus
 there is $k>0$ such that $m^k \in P\widehat{R} \cap R = \bigcap_{n>0} (P + m^n{R})$. (We are using here that $M = \sqrt{mR}$, and hence the $M$-adic and $mR$-adic topologies on $R$ coincide.) It follows that   $m^k \in P + m^{k+1}R$, which in turn implies  that $m^k \in P$  since $R $ is a local ring, a contradiction to the fact that $M = \sqrt{mR}$ and $P$ is a one-dimensional ideal of $R$. Thus  $P\widehat{R}$ has height $0$. Let $Q$ be a minimal prime ideal  of $\widehat{R}$ containing $P$. Since $\widehat{R}$ has Krull dimension $1$ and $Q$ is a minimal prime ideal of $\widehat{R}$, we have $m\widehat{R} \not \subseteq Q$ and hence $Q \cap R \subsetneq M$. Since $P \subseteq Q \cap R \subsetneq M$ and $P$ is a one-dimensional prime ideal, we conclude that $P = Q \cap R$.  
\end{proof} 

We denote by Min$(R)$ the set of minimal prime ideals of the ring $R$.  

\begin{theorem} \label{count} Let  $R$ be a generalized local ring such that $|{\rm Min}(R)| =  |{\rm Min}(\widehat{R})|$. Then  
the maximal ideal  of $R$ is the radical of a principal ideal if and only if $R_{red}$ is a  Noetherian ring of Krull dimension at most $1$.  
\end{theorem}

\begin{proof}  Suppose the maximal ideal $M$ of $R$ is the radical of a principal ideal. If $R_{red}$ has Krull dimension $0$, then $R_{red}$ is a field. Suppose $R_{red}$ has Krull dimension  more than $0$. Since $M$ is the radical of a principal ideal, say $M = \sqrt{mR}$,  it follows that the image of  $m$ in $R_{red}$ is a nonzerodivisor and hence $MR_{red}$ is regular. By Lemma~\ref{contraction}, each minimal prime ideal of $R$ is contracted from a minimal prime ideal of $\widehat{R}$. Thus, since $|{\rm Min}(R)| =  |{\rm Min}(\widehat{R})|$, it follows that every minimal prime ideal of $\widehat{R}$ contracts to a minimal prime ideal of $R$.  If $P$ is a nonmaximal prime ideal of $R$, then since $M$ is the radical of a principal ideal, $P$ is contained in a one-dimensional prime ideal $Q$ of $R$.  
By Lemma~\ref{contraction}, $Q$ is contracted from a minimal prime ideal of $\widehat{R}$, which by what we have established forces $Q$ to be a minimal prime ideal of $R$. Therefore, $P = Q$ and  $R$ has Krull dimension one. Since also $M$ is a finitely generated ideal,  $R_{red}$ is a Noetherian ring \cite[Corollary 1.21]{GHR}. 
Conversely, if $R_{red}$ is a Noetherian ring of Krull dimension at most $1$, then there are  finitely many nonmaximal prime ideals of $R$. By Prime Avoidance, $M$ is the radical of a principal ideal of $R$. 
 \end{proof}

\begin{corollary} \label{previous} Let  $R$ be a generalized local ring. If $\widehat{R}$ is a  one-dimensional domain, then $R$ is a  one-dimensional  Noetherian domain.
\end{corollary} 

\begin{proof} Since $\widehat{R}$ is a one-dimensional ring, Theorem~\ref{pre pre stable} implies that the maximal ideal $M$ of $R$ is the radical of a principal ideal. As a generalized local ring, $R$ embeds in $\widehat{R}$, so $R$ is a domain. Thus $|$Min$(R)| = 1 =  $ $|$Min$(\overline{R})|$. By Theorem~\ref{count}, $R$ is a one-dimensional Noetherian domain. 
\end{proof}

\begin{remark} 
{\em A reduced one-dimensional generalized local ring is a Noetherian ring \cite[Corollary 1.21]{GHR}. However, there exist  non-reduced one-dimensional generalized local rings that are not Noetherian rings; see for example \cite[Example 4.11]{OOne}. }
\end{remark}

\begin{corollary} Let  $R$ be a local ring with finitely generated maximal ideal. If $\widehat{R}$ is a  one-dimensional domain, then $R/(\bigcap_{i>0}M^i)$ is a  one-dimensional  Noetherian domain.
\end{corollary} 

\begin{proof} 
This follows from Corollary~\ref{previous}  and the fact that $R/(\bigcap_{i>0}M^i)$ is a generalized local ring with completion $\widehat{R}$.  
\end{proof}


We now prove a principal ideal theorem for generalized local rings that is similar in spirit to Krull's theorem:
 For a generalized local ring,  though a maximal ideal $M$ that is minimal over a principal ideal may not have height one, there is a height one maximal ideal lying over $M$ in an integral extension of $R$.  To prove Theorem~\ref{at least one}, we use the following lemma.

\begin{lemma} \label{main} Let $R$ be  a  local ring with finitely generated regular maximal ideal $M$ such that $M = \sqrt{mR}$ for some $m\in M$. Let $R'$ be the integral closure of $R$ in $T(M)$, let $C = \bigcap_{i} m^iR$ and $C' = \bigcap_i m^iR'$.   Then $\sqrt{C} = C' \cap R$ and  there is $k>0$ such that $(C' \cap R)^k \subseteq C$.
\end{lemma}

\begin{proof} 
%
By Lemma~\ref{Mati}, $R'$ is semilocal. By Theorem~\ref{pre stable}, $R'$ is Pr\"ufer in dimension zero. Since $\sqrt{mR'} = \Jac R'$,  
  Theorem~\ref{new lemma}(2) implies that  $C'$ is a radical ideal of $R'$. 
  Thus to prove the lemma it suffices to show that there is $k>0$ such that  $(C' \cap R)^k \subseteq C$.
 %
Since $R'$ is a subring of the integral closure of $R$ in its total quotient ring, we have that $m^nR' \cap R$ is contained in the  integral closure $\overline{m^nR}$ of $m^nR$ for each $n>0$. 
 Since $\sqrt{mR}  = M$, the ideal  $m^nR' \cap R$ is $M$-primary. Hence Lemma~\ref{bound} implies  there is $k>0$ such that for each $n>0$, $(m^nR' \cap R)^k \subseteq (\overline{m^nR})^{k} \subseteq m^nR$.  
 Therefore, for each $n>0$, we have $(C' \cap R)^k \subseteq (m^nR' \cap R)^k \subseteq m^nR$.  Consequently, 
 $(C' \cap R)^k
 \subseteq C$.
%
\end{proof}


\begin{theorem} \label{at least one} Let  $R$ be a generalized local ring whose maximal ideal $M$ is regular and  the radical of a principal ideal.  Let $R'$ be the integral closure of $R$ in $T(M)$.  
  Then each minimal prime ideal of $R$ is contained in a   height one maximal ideal of  $R'$. 
\end{theorem}

\begin{proof}
By Theorem~\ref{pre stable}, $R'$ is a semilocal ring with principal maximal ideals. 
Let $m \in R$ such that $M = \sqrt{mR}$, and let $C' = \bigcap_{i>0}m^iR'$. 
By Lemma~\ref{main},  there exists $k>0$ such that  $(C')^k = 0$.  Since $R'$ is semilocal and every zero-dimensional ideal of $R'$ is principal and regular, Theorem~\ref{new lemma}(1) implies there are finitely many  minimal prime ideals $Q_1,\ldots,Q_n$ of $R'$, and these prime ideals have the properties that   $C' = Q_1 \cap \cdots \cap Q_n$ and  each $Q_i$ is  contained a height one maximal ideal $M_i$ of $R'$. 
 Let $P$ be a minimal prime ideal of $R$. 
 Since $(Q_1 \cap \cdots \cap Q_n)^k \subseteq P$ and $P$ is minimal, there exists $i$ such that $Q_i \cap R = P$. Thus $P $ is contained in the height one maximal ideal $M_i$.
 %
\end{proof}

\begin{corollary} Let $R$ be a generalized local ring  with regular  maximal ideal $M$  such that $M$ is the radical of a principal ideal, and let $R'$ be  the integral closure of $R$ in $T(M)$.
If $R'$ is a local ring, then $R$ has Krull dimension one.
%

\end{corollary} 

\begin{proof} By Theorem~\ref{at least one}, any minimal prime ideal of $R$ is contained in a height one maximal ideal of $R'$. Since $R'$ is local, this forces $R'$ to be a one-dimensional ring. \end{proof}

\end{document}